\documentclass[]{amsart} 
\usepackage{amsmath,amsthm,amssymb,amsfonts,mathrsfs,color,hyperref, mathtools,crop, graphicx, enumitem, todonotes}
\usepackage[]{geometry}

\title[High-dimensional PDEs with Barron data]{Some observations on high-dimensional partial differential equations with Barron data}

\theoremstyle{plain}
\begingroup
\newtheorem{theorem}{Theorem}[section]
\newtheorem*{theorem*}{Theorem}
\newtheorem*{"theorem"}{``Theorem''}

\newtheorem{lemma}[theorem]{Lemma}
\endgroup

\theoremstyle{definition}
\begingroup

\endgroup

\theoremstyle{remark}
\begingroup
\newtheorem{remark}[theorem]{Remark}
\newtheorem{example}[theorem]{Example}
\endgroup 

\numberwithin{equation}{section}
\newenvironment{pde}{\left\{\begin{array}{rll} } {\end{array}\right.}

\newcommand{\N}{\mathbb N}

\newcommand{\R}{\mathbb R} 

\renewcommand{\P}{{\mathbb P}}

\newcommand{\dist}{{\rm dist}}

\renewcommand{\div}{{\rm div}}

\newcommand{\W}{{\mathcal W}}

\renewcommand{\d}{\mathrm{d}}

\newcommand{\dy}{\,\mathrm{d}y}
\newcommand{\dz}{\,\mathrm{d}z}
\newcommand{\ds}{\,\mathrm{d}s}

\newcommand{\dr}{\,\mathrm{d}r}
\newcommand{\ReLU}{\mathrm{ReLU}}

\newcommand{\eps}{\varepsilon}
\newcommand{\average}{{\mathchoice {\kern1ex\vcenter{\hrule height.4pt
width 6pt depth0pt} \kern-9.7pt} {\kern1ex\vcenter{\hrule
height.4pt width 4.3pt depth0pt} \kern-7pt} {} {} }}

\newcommand\showlabel{\addtocounter{equation}{1}\tag{\theequation}}

\newcommand{\B}{\mathcal{B}}

\newcommand{\sw}[1]{\textcolor{black}{#1}}

\allowdisplaybreaks

 \makeatletter
\@namedef{subjclassname@2020}{2020 Mathematics Subject Classification}
\makeatother

\begin{document}

\author{Weinan E}
\address{Weinan E\\
Department of Mathematics and Program for Applied and Computational Mathematics\\
Princeton University\\
Princeton, NJ 08544
}
\email{weinan@math.princeton.edu}

\author{Stephan Wojtowytsch}
\address{Stephan Wojtowytsch\\
Program for Applied and Computational Mathematics\\
Princeton University\\
Princeton, NJ 08544
}
\email{stephanw@princeton.edu}

\date{\today}


\maketitle

\begin{abstract}
We use explicit representation formulas to show that solutions to certain partial differential equations lie in Barron spaces or multilayer spaces if the PDE data lie in such function spaces. Consequently, these solutions can be represented efficiently using artificial neural networks, even in high dimension. Conversely, we present examples in which the solution fails to lie in the function space associated to a neural network under consideration.
\end{abstract}

\section{Introduction}
Neural network-based machine learning has had impressive success in solving very high dimensional  PDEs, see e.g.\ 
\cite{weinan2017deep, weinan2018deep, han2018solving, sirignano2018dgm, chaudhari2018deep, wang2018deepmd, raissi2019physics, bhattacharya2020model, mao2020physics, kissas2020machine, sun2020neupde, chen2020comparison}
 and many others. 
 The fact that such high dimensional PDEs can be solved with relative ease suggests that the solutions of these PDEs must have
 relatively low complexity in some sense.  The important question that we need to address is: How can we
 quantify such low complexity?
 
 In low dimension, we are accustomed to using smoothness to quantify complexity. Roughly speaking,
 a function has low complexity if it can be easily approximated using polynomials or piecewise polynomials,
 and the convergence rate of the polynomial or piecewise polynomial approximation is characterized by the 
 smoothness of the target function.  A host of function spaces have been proposed, such as the Sobolev
 spaces, Besov spaces, H\"older- and $C^k$ spaces, etc, to quantify the smoothness of functions. 
 Based on these, the analysis of numerical
 approximations to solutions of PDEs can be carried out in three steps:
 
 \begin{enumerate}
 \item Approximation theory: relating the convergence rate to the function class that the target function belongs to. 
 This stage is purely about function approximation. No PDEs are involved.
 \item Regularity theory: characterizing the function class that solutions of a given PDE  belong to. This is pure PDE theory.
 \item Numerical analysis: proving that a numerical algorithm is convergent, stable, and characterize its order of accuracy.
 \end{enumerate}
 
 We aim to carry out the same program in high dimension, with polynomials and piecewise polynomials replaced
 by neural networks. In high dimension, the most important question is whether there is a ``curse of 
 dimensionality'', i.e. whether the convergence rate deteriorates as the dimensionality is increased.
 
 In this regard, our gold standard is the Monte Carlo algorithm for computing expectations.
 The task of approximation theory is to establish similar results for function approximation for machine learning models,
 namely, given a particular machine learning model, one would like to identify the class of functions for which the given 
 machine learning model can approximate with dimension-independent convergence rate.
 It is well known that the number of parameters of a neural network to approximate a general $C^k$-function in dimension $d$ in the $L^\infty$-topology to accuracy $\eps$ generally scales like $\eps^{-\frac{d}k}$ \cite{yarotsky2017error} (up to a logarithmic factor). 
Classical (H\"older, Sobolev, Besov, \dots) function spaces thus face the ``curse of dimensionality'' and therefore are not
suited for our purpose.
 
 However, Monte Carlo-like convergence rates can indeed be established for the most interesting machine learning models
 by identifying the appropriate function spaces.
 For the random feature model, the appropriate function spaces are the extensions of the classical reproducing kernel 
 Hilbert space \cite{review_article}.
 For two-layer neural networks, those function spaces are the relatively well-studied {\em Barron spaces} \cite{weinan2019lei,barron_new}, while for deeper neural networks,  {\em tree-like function spaces} \cite{deep_barron} 
 seem to provide good candidates (at least for ReLU-activated networks). These spaces combine two convenient properties: Functions can be approximated well using a small set of parameters, and few data samples suffice to assess whether a function is a suitable candidate as the solution to a minimization problem. 

This paper is concerned in the second question listed above in high dimension.
For the reasons stated above, we must consider the question whether the solution of a PDE lie in a (Barron or tree-like) function space if the data lie in such a space. In other words, if the problem data are represented by a neural network, can the same be said for the solution, and is the solution network of the same depth as the data network or deeper? We believe that these should be the among the principal considerations of a modern regularity theory for high-dimensional PDEs. We note that a somewhat different approach to a similar question is taken in  \cite{grohs2018proof, hutzenthaler2020proof}.

In this paper, we show that solutions to three prototypical PDEs (screened Poisson equation, heat equation, viscous Hamilton-Jacobi equation) lie in appropriate function spaces associated with neural network models. These equations are considered on the whole space and the argument is based on explicit representation formulas. In bounded domains on the other hand, we show that even harmonic functions with Barron function boundary data may fail to be Barron functions, and we discuss obstacles in trying to replicate classical regularity theory in spaces of neural networks. 

While we do not claim great generality in the problems we treat, we cover some important special cases and counterexamples. As a corollary, equations whose solutions lie in a (Barron or tree-like) function space associated to neural networks cannot be considered as fair benchmark problems for computational solvers for general elliptic or parabolic PDEs: If the data are represented by a neural network, then so is the exact solution, and approximate solutions converge to the analytic solution at a dimension-independent rate as the number of parameters approaches infinity. The performance is therefore likely to be much better in this setting than in potential applications without this high degree of compatibility. 

The article is structured as follows. In the remainder of the introduction, we give a brief summary of some results concerning Barron space and tree-like function spaces. In Section \ref{section prototypes}, we consider the Poisson equation, screened Poisson equation, heat equation, and viscous Hamilton-Jacobi equation on the whole space. In two cases, we show that solutions lie in Barron space and in one case, solutions lie in a tree-like function space of depth four. In Section \ref{section boundary}, we consider equations on bounded domains and demonstrate that boundary values can make a big difference from the perspective of function representation. We also discuss the main philosophical differences between classical function spaces and spaces of neural network functions, which a novel regularity theory needs to account for.

\subsection{Previous work}

\sw{
The class of partial differential equations is diverse enough to ensure that no single theory captures the properties of all equations, and the existing literature is too large to be reviewed here. Even  the notion of what a solution is depends on the equation under consideration, giving rise to concepts such as weak solutions, viscosity solutions, entropy solutions etc. The numerical schemes to find these solutions are as diverse as the equations themselves. As a rule, many methods perform well in low spatial dimension $d\in \{1,2,3\}$, which captures many problems of applied science. Other problems, such as the Boltzmann equation, Schr\"odinger equations for many particle systems, or Black-Scholes equations in mathematical science, are posed in much higher dimension, and require different methods.
}

\sw{
In low dimension, elliptic and parabolic equations are often considered in Sobolev spaces. Finite element methods, which find approximate solutions to a PDE in finite-dimensional ansatz spaces, are empirically successful and allow rigorous a priori and a posteriori error analysis. The ansatz spaces are often chosen as spaces of piecewise linear functions on the triangulation of a domain. As a grid based method, the curse of dimensionality renders this approach unusable when the dimension becomes moderately high.
}

\sw{
Evading the curse of dimensionality when solving high-dimensional PDEs requires strong assumptions. If the right hand side even of the Poisson equation is merely in $C^{0,\alpha}$ or $L^2$, the solution cannot generally be more regular than $C^{2,\alpha}$ or $H^2$. These spaces are so large (e.g.\ in terms of Kolmogorov width) that the curse of dimensionality cannot be evaded.  
}

\sw{
Based on a hierarchical decomposition of approximating spaces, sparse grid methods can be used to solve equations in higher dimension. In their theoretical analysis, the role of classical Sobolev spaces is taken by Sobolev spaces with dominating mixed derivatives (imagine spaces where $\partial_1\partial_2u$ is controlled, but $\partial_1\partial_1u$ is not). These methods have been used successfully in medium high dimension, although the a priori regularity estimates which underly classical finite element analysis are less developed here to the best of our knowledge. An introduction can be found e.g.\ in \cite{garcke2012sparse}.
}

\sw{
Function classes of neural networks have a non-linear structure, which allows them to avoid the curse of dimensionality in approximating functions in large classes where {\em any} linear model suffers from the CoD, see e.g.\ \cite{barron1993universal}. Following their success in data science, it is a natural question whether they have the same potential in high-dimensional PDEs.
}

\sw{
It has been observed by \cite{grohs2018proof,jentzen2018proof,hutzenthaler2020proof,hornung2020space,darbon2020overcoming} that deep neural networks can approximate (viscosity) solutions to different types of partial differential equations without the curse of dimensionality, assuming that the problem data are given by neural networks as well. Generally, proofs of these results are based on links between the PDE and stochastic analysis, and showing that solutions to SDEs can be approximated sufficiently well by neural networks with the Monte-Carlo rate, even in high dimension. A more extensive list of empirical works can be found in these references.
}

\sw{
In this article, we follow a similar philosophy of explicit representation, but we consider a more restricted and technically much simpler setting. In this setting, we prove stronger results:
\begin{enumerate}
\item The neural networks we consider have one hidden layer (linear PDEs) or three hidden layers (viscous Hamilton-Jacobi equation), and are therefore much shallower than the deep networks considered elsewhere.
\item In certain cases, solutions to the PDEs are in Barron space. This not oly implies that they can be approximated without the CoD, but also that integral quantities can be estimated well using few samples, even if the solution of the PDE depends on these data samples. This follows from the fact that the unit ball in Barron space has low Rademacher complexity.
\item On the other hand, we show by way of counterexample that the solution of the Laplace equation in the unit ball with Barron boundary data is generally {\em not} given by a Barron function, possibly shedding light on the requirement of depth.
\end{enumerate}
}

\sw{
Our results can be viewed more in the context of function-space based regularity theory, whereas existing results directly address the approximation of solutions without an intermediate step.
}

\subsection{A brief review of Barron and tree-like function spaces}

Two-layer neural networks can be represented as 
\begin{equation}\label{eq mean field network}
\sw{f_m}(x) = \frac1m\sum_{i=1}^m a_i\,\sigma(w_i^Tx+b_i).
\end{equation}
The parameters $(a_i, w_i,b_i)\in \R\times \R^d\times \R$ are referred to as the weights of the network and $\sigma:\R\to\R$ as the activation function. In some parts of this article, we will focus on the popular ReLU-activation function $\sigma(z) = \max\{z,0\}$, but many arguments go through for more general activation $\sigma$. 

In the infinite width limit for the hidden layer, the normalized sum is replaced by a probability integral
\begin{equation}
f_\pi(x) = \int_{\R^{d+2}} a\,\sigma(w^Tx+b)\,\pi(\d a\otimes \d w\otimes \d b) \sw{= \mathbb E_{(a,w,b)\sim\pi} \big[a\,\sigma(w^Tx+b)}\big]
\end{equation}
On this space, a norm is defined by
\[
\|f\|_\B = \inf \left\{\int_{\R^{d+2}} |a|\,\big[|w|+ |b|\big]\,\pi(\d a\otimes \d w\otimes \d b) : \pi\text{ s.t. } f=f_\pi\right\}.
\]
The formula has to be modified slightly for non-ReLU activation functions \cite{li2020complexity} and \cite[Appendix A]{approximationarticle}. If for example $\lim_{z\to \pm \infty}\sigma(z) = \pm 1$, we may choose
\[
\|f\|_\B = \inf \left\{\int_{\R^{d+2}} |a|\,\big[|w|+ 1\big]\,\pi(\d a\otimes \d w\otimes \d b) : \pi\text{ s.t. } f=f_\pi\right\}
\]
without strong dependence on the bias variable $b$. We give some examples of functions in Barron space or not in Barron space below.

The function space theory for multi-layer networks is more complicated, and results are currently only available for ReLU activation. We do not go into detail concerning the scale of tree-like function spaces $\W^L$ associated to multi-layer neural networks here, but note the following key results.
\begin{enumerate}
\item $\W^1=\B$, i.e.\ the tree-like function space of depth $1$ coincides with Barron space. 
\item If $f:\R^k\to\R$ and $g:\R^d\to \R^k$ are functions in the tree-like function spaces $\W^\ell$ and $\W^L$ respectively (componentwise), then the composition $f\circ g:\R^d\to \R$ is in the tree-like function space $\W^{L+\ell}$ and the estimate
\[
\|f\circ g\|_{\W^{L+\ell}} \leq \max_{1\leq i\leq k} \|g_i\|_{\W^L} \,\|f\|_{\W^\ell}
\]
holds.
\item In particular, the product of two functions $f,g\in \W^L$ is generally not in $\W^L$, but in $\W^{L+1}$ since the map $(x,y)\mapsto xy$ is in Barron space (on bounded sets in $\R^2$).
\end{enumerate}

All results can be found in \cite{deep_barron}. We recall the following properties of Barron space:

\begin{enumerate}
\item If $f\in H^s(\R^d)$ for $s>\frac{d}2+2$, then $f\in \B$, i.e.\ sufficiently smooth functions admit a Barron representation -- see \cite[Theorem 3.1]{barron_new}, based on an argument in \cite[Section IX]{barron1993universal}.
\item Barron space embeds into the space of Lipschitz-continuous functions.
\item If $f\in \B$, then $f= \sum_{i=1}^\infty f_i$ where $f_i (x) = g_i(P_ix+b_i)$
\begin{enumerate}
\item $g_i$ is $C^1$-smooth except at the origin,
\item $P_i$ is an orthogonal projection on a $k_i$-dimensional subspace for some $0\leq k_i\leq d-1$,
\item $b_i$ is a shift vector.
\end{enumerate}
The proof can be found in \cite[Lemma 5.2]{barron_new}.
\item For a Barron function $f$ and a probability measure $\P$ on $[-R,R]^d$, there exist $m$ parameters $(a_i, w_i, b_i)\in \R^{d+2}$ \sw{and a two-layer neural network $f_m$ as in \eqref{eq mean field network}} such that
\begin{equation}
\|f-f_m\|_{L^2(\P)}\leq \frac{\max\{1,R\}\,\|f\|_\B}{\sqrt m}
\end{equation}
or
\begin{equation} 
\|f-f_m\|_{L^\infty[-R,R]}\leq \frac{d\,\max\{1,R\}\,\|f\|_\B}{\sqrt m}
\end{equation}
or 
\begin{equation}
\|f-f_m\|_{H^1((0,1)^d)} \leq \|f\|_\B \sqrt{\frac{d+1}m}.
\end{equation}
The proof in the Hilbert space cases is based on the Maurey-Jones-Barron Lemma (e.g. \cite[Lemma 1]{barron1993universal}), whereas the proof for $L^\infty$-approximation uses Rademacher complexities (see e.g.\ \cite{review_article}).
\end{enumerate}

\begin{example}
Examples of Barron functions are 
\begin{enumerate}
\item the single neuron activation function $f(x) = a\,\sigma(w^Tx+b)$, 
\item the $\ell^2$-norm function
\[
f(x) = c_d \int_{S^{d-1}} \sigma(w^Tx)\,\pi^0(\d w)
\]
which is represented by the uniform distribution on the unit sphere (for ReLU activation), and
\item any sufficiently smooth function on $\R^d$.
\end{enumerate}
Examples of functions which are not Barron are all functions which fail to be Lipschitz continuous and functions which are non-differentiable on a set which is not compatible with the structure theorem, e.g. $f(x) = \max\{x_1,\dots, x_d\}$ and $f(x)= \dist_{\ell^2}(x, S^{d-1})$.
\end{example}

\begin{remark}
Define the auxiliary norm
\[
\|f\|_{aux} := \int_{\R^d} |\hat f|(\xi)\cdot |\xi|\d\xi
\]
on functions $f:\R^d\to\R$. According to \cite{barron1993universal}, the estimate $\|f\|_\B \leq \|f\|_{aux}$ holds, and many early works on neural networks used $\|\cdot\|_{aux}$ in place of the Barron norm. Unlike the auxiliary norm, the modern Barron norm is automatically adaptive to the activation function $\sigma$ and Barron space is much larger than the set of functions on which the auxiliary norm is finite (which is a separable subset of the non-separable Barron space). 

Most importantly, the auxiliary norm is implicitly dimension-dependent. If $g:\R^k\to \R$ is a Barron function and $f:\R^d\to\R$ is defined as $f(x) = g(x_1,\dots, x_k)$ for $k< d$, then the auxiliary norm of $f$ is automatically infinite (since $f$ does not decay at infinity). Even when considering bounded sets and extension theorems, the auxiliary norm is much larger than the Barron norm, which allows us to capture the dependence on low-dimensional structures efficiently. 
\end{remark}

\section{Prototypical equations}\label{section prototypes}

In this article, we study four PDEs for which explicit representation formulas are available. The prototypical examples include a linear elliptic, linear parabolic, and viscous Hamilton-Jacobi equation.
The key-ingredient in all considerations is a Green's function representation with a translation-invariant and rapidly decaying Green's function.

\subsection{The screened Poisson equation}

The fundamental solution $G(x) = c_d\,|x|^{2-d}$ of $-\Delta G = \delta_0$ on $\R^d$ (for $d>2$) decays so slowly at infinity that we can only use it to solve the Poisson equation 
\begin{equation}\label{eq Laplace}
-\Delta u = f
\end{equation}
if $f$ is compactly supported (or at least decays rapidly at infinity), since otherwise the convolution integral fails to exist. Neither condition is particularly compatible with the superposition representation of one-dimensional profiles which is characteristic of neural networks.
As a model problem for second order linear elliptic equations, we therefore study the screened Poisson equation
\begin{equation}\label{eq screened Laplace}
(-\Delta + \lambda^2) u = f
\end{equation}
on the whole space $\R^d$ for some $\lambda>0$. \sw{Solutions are given explicitly as convolutions with the fundamental solution of \eqref{eq screened Laplace} and -- for large dimension -- as superpositions of one-dimensional solutions..}

\begin{lemma}\label{lemma screened Laplace}
\sw{
Assume that $f\in \B$ in dimension $d$ and $u$ solves \eqref{eq screened Laplace}. Then
\begin{equation}
\|u\|_\B \leq \lambda^{-2}\, \|f\|_\B
\end{equation}
if $\sigma$ has finite limits at $\pm \infty$ and
\begin{equation}
\|u\|_\B \leq \big[\lambda^{-2} + 2\,\lambda^{-3}\big]\,\|f\|_\B.
\end{equation}
if $\sigma = \mathrm{ReLU}$.
}
\end{lemma}

\begin{proof}[Proof of Lemma \ref{lemma screened Laplace}]
\sw{
For $d=3$, the fundamental solution $G(x) = \frac{e^{-\lambda|x|}}{4\pi\,|x|}$ of the screened Poisson equation \eqref{eq screened Laplace} is well-known. Observe that
\begin{align*}
\frac1{4\pi} \int_{\R^3} e^{-\lambda |y|}{|y|}\dy &= \int_0^\infty e^{-\lambda r} r^{-1}\,r^{3-1}\dr\\
	&= \lambda^{-2}\int_0^\infty r\lambda r\,e^{-\lambda r} \,\lambda\dr\\
	&= \lambda^{-2}\\
\frac1{4\pi} \int_{\R^3} e^{-\lambda |y|} \dy &= 2\,\lambda^{-3}.
\end{align*}
 Thus, if $d=3$ and $f(x) = a\,\sigma(w_1x_1 +b)$ we have
\begin{align*}
u(x) &= \frac{a}{4\pi} \int_{\R^3} \frac{\sigma(w_1(x_1-y_1) + b)}\,\frac{e^{-\lambda\,|y|}}{|y|}\dy\\
\|u\|_\B &\leq |a|\,\int_{\R^3} \big\{|w_1+|b|\big\}| \,\frac{e^{-\lambda\,|y|}}{|y|} + |w_1|\,|y|\,\frac{e^{-\lambda\,|y|}}{|y|}\dy\\
	&\leq \frac{|a|\big\{|w_1+|b|\big\}}{\lambda^2} + \frac{2\,|a|\,|w_1|}{\lambda^3}.
\end{align*}
If $\sigma$ is bounded, the bias term is replaced by $1$. 
}

\sw{
In arbitrary dimension, the unit ball in Barron space is the closed convex hull of functions of the form $f(x) = a\,\sigma(w^Tx+b)$ such that $|a|\,\big[|w|+|b|\big]\leq 1$ (ReLU case). When considering the Euclidean norm on $\R^d$, after a rotation, we note that the solution $u$ of \eqref{eq screened Laplace} with right hand side $a\,\sigma(w^T\cdot +b)$ satisfies
\[
\|u\|_\B \leq |a|\,\{|w| + |b|\}\left[\lambda^{-2} + 2\,\lambda^{-3}\right]
\]
since $u$ only depends on $w^Tx$ and is constant in all directions orthogonal to $w$. The estimate follows in the general case by linearity.
}
\end{proof}

\subsection{The Poisson equation}\label{section laplace}

While the importance of the forcing term $\lambda u$ in the argument above is clear, it is important to ask whether the term is in fact necessary or just convenient. We show that in many cases, it is the \sw{latter}, at least in naive models. Consider the equation
\begin{equation}
-\Delta u = \ReLU(x_1)
\end{equation}
which is solved by
\begin{equation}
u(x) = -\frac{\max\{0,x_1\}^3}6 + h(x), \qquad \Delta h(x) = 0.
\end{equation}
Since any function $u$ in ReLU-Barron space grows at most linearly at infinity, $\frac{\max\{0,x_1\}^3}6$ is not a ReLU-Barron function. This fast growth cannot be compensated by the free harmonic function $h$. Note that $h$ could grow at most cubically at infinity and would therefore have to be a cubic polynomial by Liouville's theorem \cite{boas1988short}. Then
\[
\lim_{t\to \infty} \frac{h(te_1)}{h(-te_1)} = \pm 1
\]
depending on the degree of $h$, so $u$ could not be bounded by a subcubic function in both half spaces $\{x_1>0\}$ and $\{x_1<0\}$. A very similar argument goes through for activation functions $\sigma$ for which the limits $\lim_{z\to \pm \infty} \sigma(z)$ exist and are finite (with quadratic polynomials instead of cubics). Notably, activation functions $\sigma \in \{\sin,\cos,\exp\}$ do not suffer the same obstacle since there second anti-derivative grows at the same rate as the function itself. Considering these activation functions would likely prove fruitful, but take us back into the classical field of Fourier analysis. We pursue a different route and allow the right hand side and solution of a PDE to be represented by neural networks with different activation functions.

Consider a $C^2$-smooth activation function $\sigma$. Then
\[
\Delta [a\,\sigma(w^Tx+b)] = a\,|w|^2\,\sigma''(w^Tx+b).
\]
In particular, for $0<\alpha<1$ we have
\[
\left\|\Delta [a\,\sigma(w^Tx+b)] \right\|_{L^\infty(\R^d)} = |a|\,|w|^2\,\|\sigma''\|_{L^\infty}, \qquad \left[\Delta [a\,\sigma(w^Tx+b)] \right]_{C^{0,\alpha}(\R^d)} = |a|\,|w|^{2+\alpha}[\sigma'']_{C^{0,\alpha}}.
\]
For ReLU activation and functions in Barron space, the Laplacian of $a\,\sigma(w^Tx+b)$ is merely a (signed) Radon measure. On the other hand, if $\sigma \in C^{2,\alpha}(\R)$ is bounded and has bounded first and second derivatives, then we can make sense of the Laplacian in the classical sense. The same is true for the superposition
\[
f_\pi(x) = \int_{\R\times\R^d\times \R} a\,\sigma(w^Tx+b)\,\pi(\d a\otimes \d w\otimes \d b)
\]
if the modified Barron norm
\[\label{eq modified Barron norm}
\int |a|\,[|w|^{2+\alpha}+1]\,\pi(\d a \otimes \d w \otimes \d b)\showlabel
\]
is finite. Unlike applications in data science, most settings in scientific computing require the ability to take derivatives, at least in a generalized sense. For elliptic problems, it seems natural to consider neural networks with different activation functions to represent target function and solution to the PDE:
\begin{align*}
f(x) &= \int a\,\sigma''(w^Tx+b)\,\d\pi\\
u(x) &= \int \tilde a \,\sigma(\tilde w+\tilde b)\,\d\tilde\pi\\
\end{align*}
where
\[
\int|\tilde a|\,[|\tilde w|^{2+\alpha} + 1]\d\tilde\pi<\infty, \qquad\int \frac{|a|}{|w|^2}[|w|^{2+\alpha}+1]\d\pi<\infty.
\]
All considerations were in the simplest setting of the Laplace operator on the whole space. We will discuss the influence of boundary values below.

\subsection{The heat equation}

Parabolic equations more closely resemble the screened Poisson equation than the Poisson equation: The fundamental solution decays rapidly at infinity, and no fast growth at infinity is observed. The (physical) solution $u$ does not grow orders of magnitude faster than the initial condition $u_0$ at any positive time. Moreover, the heat kernel is a probability density for any time and in any dimension, so no dimension-dependent factors are expected.

As a prototypical example of a parabolic equation, we consider the heat equation
\begin{equation}
\begin{pde}
(\partial_t-\Delta) u &=f &t>0\\
u &= u_0&t=0
\end{pde}.
\end{equation}
The solution is given as the superposition of the solution of the homogeneous equation with non-zero initial values and the solution of the inhomogeneous solution with zero initial value:
\begin{align*}
u(t,x) &= u_{hom}(t,x) + u_{inhom}(t,x)\\
	&= \frac{1}{(4\pi t)^{d/2}}\int_{\R^d} u_0(y)\,\exp\left(-\frac{|x-y|^2}{4t}\right)\,\dy\\
	&\hspace{2cm}+ \int_0^t\frac{1}{(4\pi (t-s))^{d/2}}\int_{\R^d}f(s,y)\,\exp\left(-\frac{|x-y|^2}{4(t-s)}\right)\,\dy.
\end{align*}

\begin{lemma}\label{lemma heat equation}
\sw{
Lef $f$ be a ReLU-Barron function of $t$ and $x$. Then $u_{hom}$ is a Barron function of $x$ and $\sqrt{t}$ with Barron norm 
\begin{align*}
\|u_{hom}\|_\B &\leq 2\,\|f\|_\B
\end{align*}
For fixed time $t>0$, $u_{hom}$ and $u_{inhom}$ are Barron functions of $x$ with
\[
\|u_{hom}(t,\cdot)\|_\B \leq \big[1+\sqrt{t}\big], \qquad \|u_{hom}(t,\cdot)\|_\B\leq \left[t + \frac23\,t^{3/2}+ \frac{t^2}2 + \frac25\,t^{5/2}\right]\,\|f\|_\B.
\]
}
\end{lemma}

\sw{The proof is a direct calculation using the Green's function. It is presented in the appendix.}

\begin{remark}
\sw{
The dependence on $\sqrt{t}$ is a consequence of the parabolic scaling. The exponent $t^{5/2}$ occurs because the Barron norm of $f(t,\cdot)$ as a function of $x$ scales like $\max\{1,t\}$ in the worst case (the time variable, now a constant, is absorbed into the bias). A further factor of $t$ stems from the increasing length of the interval over which the time-integral in the definition of $u_{inhom}$ is given.
}

\sw{
The same argument applies under the weaker assumption that the source term $f$ is in the Bochner space $L^\infty\big((0,\infty); \B(\R^d)\big)$ with the stronger bound $\|u_{inhom}(t, \cdot)\|_\B \leq C\big[t + t^{3/2}\big]$. Also in Barron-spaces for bounded and Lipschitz-continuous activation, the exponent $5/2$ can be lowered to $3/2$ since the size of the bias term does not enter the Barron norm. 
}
\end{remark}

We therefore can efficiently solve the homogeneous heat equation for a Barron initial condition in both space and time using two-layer neural networks. For the heat equation with a source term, we can solve efficiently for the heat distribution at a fixed finite time $t>0$.

\subsection{A viscous Hamilton-Jacobi equation}

All PDEs we previously considered are linear, and we showed that if the right hand side is in a suitable Barron space, then so is the solution. This is non-surprising as the solution of one of these equations with a ridge function right hand side has a ridge function solution. By linearity, the same is true for superpositions of ridge functions. This structure is highly compatible with two-layer neural networks, which are superpositions of a specific type of ridge function. The linear equations are invariant under rescalings (and rotations), so all problems are reduced to ODEs or 1+1-dimensional PDEs.

The same argument cannot be applied to non-linear equations. For our final example, we study the viscous Hamilton-Jacobi equation
\begin{equation}\label{eq viscous Hamilton-Jacobi}
u_t - \Delta u + |\nabla u|^2 = 0,
\end{equation}
\[
(\partial_t - \Delta)\,u = \frac{v_t}{-v} -\div\left(\frac{\nabla v}{-v}\right) = \frac{v_t}{-v} - \frac{\Delta v}{-v} 
\]
for which an explicit formula is available. If $v$ solves $v_t - \Delta v = 0$, then $u= \log (-v)$ solves
\[
(\partial_t - \Delta) u = -\frac{v_t}v - \div\left(\frac{\nabla v}{-v}\right) = -\frac{(\partial_t-\Delta) u}v - \frac{|\nabla v|^2}{v^2} = 0 - \big|\nabla \log (-v)\big|^2 = -|\nabla u|^2.
\]
Thus the solution $u$ of \eqref{eq viscous Hamilton-Jacobi}
with initial condition $u(0,\cdot) = u_0$ is given by
\begin{align*}
u(t,x) &= -\log\left(\frac1{(4\pi t)^{d/2}} \int_\R \exp\left(-\frac{|x-y|^2}{4 t}\right)\exp\left(-u_0(y)\right)\dy\right).
\end{align*}

\begin{lemma}
Assume that $u_0: \R^d\to [\beta_-,\beta_+]$ is a Barron function. Then the solution $u$ of \eqref{eq viscous Hamilton-Jacobi} as a function of $\sqrt{t}$ and $x$ is in the tree-like function space with $3$ hidden layers and
\[
\|u\|_{\W^{3}} \leq \exp(\beta_+-\beta_-) \,\|u_0\|_{\B}.
\]
\end{lemma}

\begin{proof}
Recall that due to \cite[Section 4.1]{barron_new} we have
\begin{enumerate}
\item $\exp:[\beta_-,\beta_+]\to\R$ is a Barron function with Barron norm $\leq \exp(\beta_+) - \exp(\beta_-) \leq \exp(\beta_+)$, and
\item $\log: [\gamma_-,\gamma_+]\to \R$ is a Barron function with Barron norm $\frac1{\gamma_-} - \frac1{\gamma_+} \leq \frac1{\gamma_-}$ for any $0<\gamma_-<\gamma_+ < \infty$.
\end{enumerate}
The estimates in this precise form hold for ReLU activation, but similar estimates can easily be obtained for more general $\sigma$. However, a function space theory is not yet available for more general $\sigma$.

If $-u_0:\R^d\to [\beta_-, \beta_+]$ has Barron norm $\|u_0\|_{\B(\R^d)}\leq C_0$, then $\exp(-u_0)$ is a function with tree-like three-layer norm $\leq \exp(\beta_+)C_0$. Using the same change of variables as for the heat equation, we find that 
\[
F(t,x) =  \frac1{(4\pi t)^{d/2}} \int_\R \exp\left(-\frac{|x-y|^2}{4 t}\right)\exp\left(-u_0(y)\right)\dy
\]
is a three-layer tree-like function of $\sqrt t$ and $x$. We conclude that $u$ is a tree-like four-layer function of $\sqrt t$ and $x$ with norm
\[
\|u\|_{\W^3(\R^d)} \leq \frac{1}{\exp(\beta_-)}\,\exp(\beta_+)\,C_0 = \exp(\beta_+-\beta_-)\,C_0.
\]
\end{proof}

So a viscous Hamilton-Jacobi equation whose initial condition is a bounded Barron function can be solved using a four-layer ReLU-network (but the parameters may be very large if the oscillation of $u_0$ is not small).

\section{On the influence of boundary values}\label{section boundary}

\subsection{A counterexample on the unit ball}
In all examples above, a critical ingredient was the translation invariance of the PDE, which breaks down if we consider bounded domains or non-constant coefficients. When solving $-\Delta u= 0$ with boundary values $f(x_1)$ on a bounded domain, the solution never depends only on the variable $x_1$ unless $f$ is an affine function. The Barron space theory of PDEs on bounded domains is therefore markedly different from the theory in the whole space.

\begin{lemma}
Let $\sigma(z) = \max\{z,0\}$ be ReLU-activation and $g(x) = \sigma(x_1)$ a Barron function on $\R^d$ for $d\geq 2$. Denote by $B^d$ the unit ball in $\R^d$ and denote by $u$ the solution to the PDE
\begin{equation}
\begin{pde}
-\Delta u &= 0 &\text{in }B^d\\
u &= g &\text{on }\partial B^d
\end{pde}.
\end{equation}
If $d\geq 3$, $u$ is not a Barron function on $B^d$.
\end{lemma}

\begin{proof}
Assume for the sake of contradiction that $u$ is a Barron function and $d\geq 3$. If $u$ is a Barron function, then $u$ is defined (non-uniquely) on the whole space $\R^d$ by the explicit representation formula. We observe that $u\in C^\infty_{loc}(B^d)$ and that $\partial_1u = \partial_1g$ is discontinuous on the equatorial sphere $\partial B^d \cap \{x_1=0\}$ since $e_1$ is tangent to the unit sphere at this point. Thus $\Sigma:= \partial B^d \cap \{x_1=0\}$ is contained in the countable union of affine spaces on which $u$ is not differentiable. 
If $d>2$, $\Sigma$ is a $d-2$-dimensional curved submanifold of $\R^d$ and any $d-1$-dimensional subspace which does not coincide with the hyperplane $\{x_1=0\}$ intersects $\Sigma$ in a set of measure zero (with respect to the $d-2$-dimensional Hausdorff measure). Thus $u$ must be non-differentiable on the entire hyper-plane $\{x_1=0\}$, leading to a contradiction.
\end{proof}

Note that this argument is entirely specific to ReLU-activation and to two-layer neural networks. On the other hand, if $\sigma\in C^\infty$, then also $u\in C^\infty(\overline{B^d})$, and thus in particular $u\in \B$ (without sharp norm estimates). 
More generally, assume we wish to solve the Laplace equation on the unit ball in $d$ dimensions with boundary values which are a one-dimensional profile
\begin{equation}
\begin{pde}
-\Delta u(x) &= 0 &|x|<1\\
u(x) &= g(x_1) & |x|=1.
\end{pde}
\end{equation}
We may decompose $u = g+v$ where 
\begin{equation}
\begin{pde}
-\Delta v&= g''(x_1) &|x|<1\\
v &=0 &|x|=1
\end{pde}.
\end{equation}
If we abbreviate $(x_2,\dots, x_d) = \hat x$, it is clear by symmetry that $v$ only depends on $|\hat x|$, i.e.\ $v = \psi(x_1, |\hat x|)$ where
\[\showlabel
\begin{pde}\Delta \psi + \frac{d-2}{y_2}\,\partial_2\psi &= - g''(y_1) &y_1^2 + y_2^2 <1, \:y_2>0\\ \psi &= 0 &\{y_1^2 + y_2^2 =1, y_2>0\}  \end{pde}.
\]
Since solutions of the original equation are smooth, we conclude that $\partial_2\psi =0$ on $\{y_2=0\}$, meaning that also $\psi=0$ on the remaining portion $\{y_2=0\} \cap \{y_1^2+y_2^2\leq 1\}$ of the boundary.

Thus the solution of Laplace's equation with ridge function boundary values is not a ridge function, but enjoys a high degree of symmetry nonetheless. Instead of using neural networks as one-dimensional functions of $w^Tx$, it may be a more successful endeavour to consider superpositions of one-dimensional functions of the two-dimensional data
\[
 \left( w^T x,\:\: \sqrt{|x|^2 - \left|\frac{w^T}{|w|}x\right|^2}\right) = \left(w^Tx, \left\|\left(I - \frac{w}{|w|}\otimes \frac w{|w|}\right)x\right\|\right).
\]
The second component in the vector is a (ReLU-)Barron function. Thus, if $\psi:\R^2\to \R$ is a Barron function, then 
\[
u(x) = \psi \left(w^Tx, \left\|\left(I - \frac{w}{|w|}\otimes \frac w{|w|}\right)x\right\|\right)
\]
is a tree-like function of depth three. In the unit ball, this modified data can be generated explicitly, while in more complex domains, it must be learned by a deeper neural network. Whether $\psi$ is in fact a Barron function currently remains open.

\subsection{A counterexample on a ``Barron''-domain}

In the previous example, we showed that the harmonic function on a domain with Barron-function boundary values may not be a Barron function, but might be a tree-like function of greater depth. We may be tempted to conjecture the following: {\em If the boundary of a domain $U$ can locally be written as the graph of a Barron function over a suitable hyperplane and $g:\partial U\to \R$ is a Barron function, then the harmonic function $u$ on $U$ with boundary values $g$ is a tree-like function.} Such domains coincide with domains with `Barron boundary' considered in \cite{barron_boundaries}. This is generally false, as a classical counterexample to regularity theory on nonconvex Lipschitz domains shows.

Consider the planar domain
\[
D_\theta = \{x\in \R^2 : 0 < \phi < \theta\}
\]
where $\phi\in (0,2\pi)$ is angular polar coordinate and $0<\theta<2\pi$. For $k\in \N$, consider 
\[
u_{k,\theta}:D_\theta\to \R,\qquad u_{k,\theta}(r,\phi) = r^{\frac{k\pi}\theta}\,\sin\left(\frac{k\pi}\theta\,\phi\right).
\]
Observe that
\begin{itemize}
\item $-\Delta u_{k,\theta} = 0$ for all $k\in \N$, $0<\theta<2\pi$.
\item For $k=1$ and $\pi<\theta<2\pi$, we note that $u_{k,\theta}$ is not Lipschitz continuous at the origin. 
\end{itemize}
We can consider bounded domains $\widetilde D_\theta$ such that $\widetilde D_\theta \cap B_1(0)= D_\theta \cap B_1(0)$ and either 
\begin{enumerate}
\item $\partial \widetilde D_\theta$ is polygonal or
\item $\partial \widetilde D_\theta$ is $C^\infty$-smooth away from the origin.
\end{enumerate}
In either case, $\partial \widetilde D_\theta$ can be locally represented as the graph of a ReLU-Barron function. 
Since $u_{k,\theta}\equiv 0$ on $\partial D_\theta$ and $u_{k,\theta}$ is smooth away from the origin, we find that $u_{k,\theta}$ is $C^\infty$-smooth on $\partial \widetilde D_\theta$. In particular there exists a Barron function $g_{k,\theta}$ such that $g_{k,\theta}\equiv u_{k,\theta}$ on $\partial \widetilde D_\theta$. The unique solution to the boundary value problem
\[
\begin{pde}
-\Delta u &= 0 &\text{in }\widetilde D_\theta\\
u &= g_{k,\theta} &\text{on }\partial \widetilde D_\theta
\end{pde}
\]
is $u_{k,\theta}$ itself. Again, for $k=1$ and $\pi<\theta<2\pi$, this harmonic function is not Lipschitz-continuous on the closure of $\widetilde D_\phi$, and thus in particular not in any tree-like function space.

As we observed in Section \ref{section laplace}, classical Barron spaces as used in data-scientific applications behave similar to $C^{0,\alpha}$-spaces in PDE theory and another scale of $C^{2,\alpha}$-type may be useful. Such spaces may also describe more meaningful boundary regularity.

\subsection{Neural network spaces and classical function spaces}

In this section, we briefly sketch some differences between specifically Barron spaces and classical function spaces, which we expect to pose challenges in the development of a regularity theory in spaces of neural network functions.

Classical regularity theory even in the linear case uses nonlinear coordinate transformations like straightening the boundary and relies on the ability to piece together local solutions using a partition of unity. There are major differences between classical function spaces and Barron spaces. Recall that
\begin{enumerate}
\item if $u, v\in C^{2,\alpha}$, then also and $u\circ v \in C^{2,\alpha}$, and
\item  if $u\in C^{2,\alpha}$ and $a\in C^{0,\alpha}$, then also $a\,\partial_i\partial_j u\in C^{0,\alpha}$. 
\end{enumerate}
Neither property holds in Barron space in dimension $d\geq 2$. The first property is important for coordinate transformations and localizing arguments, the second when considering differential operators with variable coefficients. The properties also fail in Sobolev spaces, which is why the boundary and coefficients of a problem are typically assumed to have greater regularity than is expected of the solution or its second derivatives respectively (e.g.\ bounded measurable coefficients, $C^2$-boundaries, \dots).

While spaces of smooth functions are invariant under diffeomorphisms, Barron-space is only invariant under linear transformations: If $\phi:\R^d\to\R^d$ is a non-linear diffeomorphism, there exists a hyperplane $H =\{w^T\cdot + b=0\}$ in $\R^d$ such that $\phi^{-1}(H)$ is not a linear hyperplane. Thus the function $u(x) = \sigma \big(w^T\phi(x)+b)$ is not a Barron-function since its singular set is not straight.

Compositions of tree-like functions are tree-like functions (for deeper trees) and compositions of flow-induced functions for deep ResNets \cite{weinan2019lei} are flow-induced functions (but flow-induced function classes are non-linear). This is the first major difference between spaces of neural networks and spaces of classically smooth functions which we want to point out. While in the Sobolev setting, greater regularity is assumed, we believe that in the neural network setting, deeper networks should be considered.

The second difference is about the `locality' of the function space criterion. Note that a function $u$ is $C^{k,\alpha}$- or $H^k$-smooth on a domain $U$ if and only if there exists a finite covering $\{U_1,\dots, U_N\}$ of $U$ by open domains such that $u|_{U_k}$ is in the appropriate function space. The smoothness criterion therefore can be localized. Even in ostensibly non-local fractional order Sobolev spaces, the smoothness criterion can be localized under a mild growth or decay condition.

In general, we cannot localize the Barron property in the same way. We describe a counter-example for Barron spaces with ReLU-activation, but we expect a similar result to hold in more general function classes. Consider a U-shaped domain in $\R^2$, e.g.\ 
\[
U = \underbrace{\big((0,1)\times (0,3) \big)}_{\text{left column}} \cup \underbrace{\big((2,3)\times (0,3)\big)}_{\text{right column}} \cup\underbrace{\big( (0,3) \times (0,1)\big)}_{\text{horizontal bar}}.
\]
The function 
\[
f: U\to \R, \qquad f(x) = \begin{cases} \sigma(x_2-1) & x_1<1.5\\ 0 &x_1\geq 1.5\end{cases}
\]
is in Barron space on each of the domains $U_1 = \{x\in U : x_1< 2\}$ and $U_2 = \{x\in U : x_1>1\}$, but overall $f$ is not in Barron space since the set of points on which a Barron function is non-differentiable is a countable union of points and lines in $\R^2$. In particular, $f$ could not be differentiable on $(2,3)\times\{2\}$, but clearly is.

\bibliographystyle{../../alphaabbr}
\bibliography{../../NN_bibliography}

\newcommand{\etalchar}[1]{$^{#1}$}
\begin{thebibliography}{GHJVW18}

\bibitem[Bar93]{barron1993universal}
A.~R. Barron.
\newblock Universal approximation bounds for superpositions of a sigmoidal
  function.
\newblock {\em IEEE Transactions on Information theory}, 39(3):930--945, 1993.

\bibitem[BB88]{boas1988short}
H.~Boas and R.~Boas.
\newblock Short proofs of three theorems on harmonic functions.
\newblock {\em Proceedings of the American Mathematical Society}, pages
  906--908, 1988.

\bibitem[BHKS20]{bhattacharya2020model}
K.~Bhattacharya, B.~Hosseini, N.~B. Kovachki, and A.~M. Stuart.
\newblock Model reduction and neural networks for parametric pdes.
\newblock {\em arXiv preprint arXiv:2005.03180}, 2020.

\bibitem[CDW20]{chen2020comparison}
J.~Chen, R.~Du, and K.~Wu.
\newblock A comparison study of deep galerkin method and deep ritz method for
  elliptic problems with different boundary conditions.
\newblock {\em arXiv:2005.04554 [math.NA]}, 2020.

\bibitem[COO{\etalchar{+}}18]{chaudhari2018deep}
P.~Chaudhari, A.~Oberman, S.~Osher, S.~Soatto, and G.~Carlier.
\newblock Deep relaxation: partial differential equations for optimizing deep
  neural networks.
\newblock {\em Research in the Mathematical Sciences}, 5(3):30, 2018.

\bibitem[CPV20]{barron_boundaries}
A.~Caragea, P.~Petersen, and F.~Voigtlaender.
\newblock Neural network approximation and estimation of classifiers with
  classification boundary in a barron class.
\newblock {\em arXiv:2011.09363 [math.FA]}, 2020.

\bibitem[DLM20]{darbon2020overcoming}
J.~Darbon, G.~P. Langlois, and T.~Meng.
\newblock Overcoming the curse of dimensionality for some hamilton--jacobi
  partial differential equations via neural network architectures.
\newblock {\em Research in the Mathematical Sciences}, 7(3):1--50, 2020.

\bibitem[EHJ17]{weinan2017deep}
W.~E, J.~Han, and A.~Jentzen.
\newblock Deep learning-based numerical methods for high-dimensional parabolic
  partial differential equations and backward stochastic differential
  equations.
\newblock {\em Communications in Mathematics and Statistics}, 5(4):349--380,
  2017.

\bibitem[EMW19]{weinan2019lei}
W.~E, C.~Ma, and L.~Wu.
\newblock Barron spaces and the compositional function spaces for neural
  network models.
\newblock {\em arXiv:1906.08039 [cs.LG]}, 2019.

\bibitem[EMWW20]{review_article}
W.~E, C.~Ma, S.~Wojtowytsch, and L.~Wu.
\newblock Towards a mathematical understanding of neural network-based machine
  learning: what we know and what we don't.
\newblock {\em CSIAM Trans. Appl. Math.}, 1(4):561--615, 2020.

\bibitem[EW20a]{deep_barron}
W.~E and S.~Wojtowytsch.
\newblock On the {B}anach spaces associated with multi-layer {ReLU} networks of
  infinite width.
\newblock {\em CSIAM Trans. Appl. Math.}, 1(3):387--440, 2020.

\bibitem[EW20b]{barron_new}
W.~E and S.~Wojtowytsch.
\newblock Representation formulas and pointwise properties for {B}arron
  functions.
\newblock {\em arXiv:2006.05982 [stat.ML]}, 2020.

\bibitem[EW21]{approximationarticle}
W.~E and S.~Wojtowytsch.
\newblock Kolmogorov width decay and poor approximators in machine learning:
  Shallow neural networks, random feature models and neural tangent kernels.
\newblock {\em Res Math Sci}, 8(5), 2021.

\bibitem[EY18]{weinan2018deep}
W.~E and B.~Yu.
\newblock The deep ritz method: a deep learning-based numerical algorithm for
  solving variational problems.
\newblock {\em Communications in Mathematics and Statistics}, 6(1):1--12, 2018.

\bibitem[Gar12]{garcke2012sparse}
J.~Garcke.
\newblock Sparse grids in a nutshell.
\newblock In {\em Sparse grids and applications}, pages 57--80. Springer, 2012.

\bibitem[GHJVW18]{grohs2018proof}
P.~Grohs, F.~Hornung, A.~Jentzen, and P.~Von~Wurstemberger.
\newblock A proof that artificial neural networks overcome the curse of
  dimensionality in the numerical approximation of black-scholes partial
  differential equations.
\newblock {\em arXiv preprint arXiv:1809.02362}, 2018.

\bibitem[HJE18]{han2018solving}
J.~Han, A.~Jentzen, and W.~E.
\newblock Solving high-dimensional partial differential equations using deep
  learning.
\newblock {\em Proceedings of the National Academy of Sciences},
  115(34):8505--8510, 2018.

\bibitem[HJKN20]{hutzenthaler2020proof}
M.~Hutzenthaler, A.~Jentzen, T.~Kruse, and T.~A. Nguyen.
\newblock A proof that rectified deep neural networks overcome the curse of
  dimensionality in the numerical approximation of semilinear heat equations.
\newblock {\em SN Partial Differential Equations and Applications}, 1:1--34,
  2020.

\bibitem[HJS20]{hornung2020space}
F.~Hornung, A.~Jentzen, and D.~Salimova.
\newblock Space-time deep neural network approximations for high-dimensional
  partial differential equations.
\newblock {\em arXiv preprint arXiv:2006.02199}, 2020.

\bibitem[JSW18]{jentzen2018proof}
A.~Jentzen, D.~Salimova, and T.~Welti.
\newblock A proof that deep artificial neural networks overcome the curse of
  dimensionality in the numerical approximation of kolmogorov partial
  differential equations with constant diffusion and nonlinear drift
  coefficients.
\newblock {\em arXiv preprint arXiv:1809.07321}, 2018.

\bibitem[KYH{\etalchar{+}}20]{kissas2020machine}
G.~Kissas, Y.~Yang, E.~Hwuang, W.~R. Witschey, J.~A. Detre, and P.~Perdikaris.
\newblock Machine learning in cardiovascular flows modeling: Predicting
  arterial blood pressure from non-invasive 4d flow {MRI} data using
  physics-informed neural networks.
\newblock {\em Computer Methods in Applied Mechanics and Engineering},
  358:112623, 2020.

\bibitem[LMW20]{li2020complexity}
Z.~Li, C.~Ma, and L.~Wu.
\newblock Complexity measures for neural networks with general activation
  functions using path-based norms.
\newblock {\em arXiv preprint arXiv:2009.06132}, 2020.

\bibitem[MJK20]{mao2020physics}
Z.~Mao, A.~D. Jagtap, and G.~E. Karniadakis.
\newblock Physics-informed neural networks for high-speed flows.
\newblock {\em Computer Methods in Applied Mechanics and Engineering},
  360:112789, 2020.

\bibitem[RPK19]{raissi2019physics}
M.~Raissi, P.~Perdikaris, and G.~E. Karniadakis.
\newblock Physics-informed neural networks: A deep learning framework for
  solving forward and inverse problems involving nonlinear partial differential
  equations.
\newblock {\em Journal of Computational Physics}, 378:686--707, 2019.

\bibitem[SS18]{sirignano2018dgm}
J.~Sirignano and K.~Spiliopoulos.
\newblock Dgm: A deep learning algorithm for solving partial differential
  equations.
\newblock {\em Journal of computational physics}, 375:1339--1364, 2018.

\bibitem[SZS20]{sun2020neupde}
Y.~Sun, L.~Zhang, and H.~Schaeffer.
\newblock Neupde: Neural network based ordinary and partial differential
  equations for modeling time-dependent data.
\newblock In {\em Mathematical and Scientific Machine Learning}, pages
  352--372. PMLR, 2020.

\bibitem[WZHE18]{wang2018deepmd}
H.~Wang, L.~Zhang, J.~Han, and W.~E.
\newblock Dee{PMD}-kit: A deep learning package for many-body potential energy
  representation and molecular dynamics.
\newblock {\em Computer Physics Communications}, 228:178--184, 2018.

\bibitem[Yar17]{yarotsky2017error}
D.~Yarotsky.
\newblock Error bounds for approximations with deep {ReLU} networks.
\newblock {\em Neural Networks}, 94:103--114, 2017.

\end{thebibliography}

\newpage

\appendix
\section{Proofs}

\begin{proof}[Proof of Lemma \ref{lemma heat equation}]
Assume specifically that $u_0\in \B(\R^d)$ and $f\in \B(\R^{d+1})$. With the substitution $z=\frac{y-x}{\sqrt{t}}$ we obtain
\begin{align*}
u_{hom}(t,x) &= \frac{1}{(4\pi t)^{d/2}} \int_{\R^d} u_0(y)\,\exp\left(-\frac{|x-y|^2}{4t}\right)\dy\\
	&= \frac{1}{(4\pi)^{d/2}}\int_{\R^d} u_0(x+\sqrt{t}\,z)\,\exp(-|z|^2/4)\dz\\
	&= \frac{1}{(4\pi)^{d/2}}\int_{\R^d} \int_{\R^{d+2}} a\,\sigma\left(w^T\big(x+\sqrt{t}\,z\big)+b\right)\pi_0(\d a\otimes \d w\otimes \d b)\,\exp(-|z|^2/4)\dz\\
	&= \int_{\R^{d+2}}a\,\int_{\R^d} \sigma\left(w^T\big(x+\sqrt{t}\,z\big)+b\right) \frac{\exp(-|z|^2/4)}{(4\pi)^{d/2}}\dz\,\pi_0(\d a\otimes \d w \otimes \d b).
\end{align*}
In particular, $u$ is a Barron function of $(\sqrt{t},x)$ on $(0,\infty)\times \R^d$ with Barron norm 
\begin{align*}
\|u_{hom}\|_{\B\big((0,\infty)\times \R^d\big)} &\leq \int_{\R^{d+2}}|a|\int_{\R^d} \big[|w| + |w^Tz| + |b|\big] \frac{\exp(-|z|^2/4)}{(4\pi)^{d/2}}\dz\,\pi_0(\d a\otimes \d w \otimes \d b)\\
	&\leq 2 \|u_0\|_\B
\end{align*}
in both the ReLU and general case. There is no dimension dependence in the integral since up to rotation $|w^Tz|= |w|\,|z_1|$, so the integral reduces to the one-dimensional case. This agrees with the intuition that $u$ is the superposition of solutions of $(\partial_t-\Delta)u=0$ with initial condition given by the one-dimensional profile $\sigma(w^Tx+b)$. Instead of considering high-dimensional heat kernels, we could have reduced the analysis to $1+1$ dimensions.

The fact that $u$ is a Barron function of $\sqrt{t}$ rather than $t$ is due to the parabolic scaling of the equation. 
For fixed time $t>0$, the same argument shows that in the $x$-variable we have
\[
\|u_{hom}\|_{\B\big(\{t\}\times \R^d\big)} \leq \int_{\R^{d+2}\times\R^d}|a| \big[|w|\,+\sqrt{t}|w^Tz|+|b|\big] \,\d\pi\dz \leq \big[1+ \sqrt{t}\big] \|u_0\|_\B.
\]
The inhomogeneous part of $u$ is a superposition of Barron functions in $x$ and $\sqrt{t-s}$ for $0<s<t$, which does not allow us to express $u$ as a Barron function of both space and time in an obvious way. However, since $f(t,\cdot)$ is a Barron function in space with norm $\|f(t,\cdot)\|_{\B(\R^d)} \leq \|f\|_{\B((0,\infty)\times\R^d)}\max\{1,t\}$, we use the analysis of the homogeneous problem to obtain that
\begin{align*}
\big\|u_{inhom}(t,\cdot)\big\|_{\B(\R^d)} &\leq \int_0^t \|f(s,\cdot)\|_{\B(\R^d)}\,\big[1+\sqrt{t-s}\big]\ds\\
	&\leq  \|f\|_{\B((0,\infty)\times\R^d)} \int_0^t \max\{1,s\}\,\big[1+\sqrt{t-s}\big]\ds\\
	&= \left[t+ \frac 23 \,t^{3/2} + \frac{t^2}2 + \frac25\,t^{5/2}\right]\,\|f\|_{\B((0,\infty)\times\R^d)}.
\end{align*}
The Barron norm of the combined solution $u(t,\cdot) = u_{hom}(t,\cdot) + u_{inhom}(t,\cdot)$ therefore grows at most approximately like $t^{5/2}$ in time, independently of dimension $d$.
\end{proof}

\end{document}